\pdfoutput=1
\RequirePackage{ifpdf}
\ifpdf
\documentclass[pdftex]{sigma}
\else
\documentclass{sigma}
\fi

{ \theoremstyle{definition}
\newtheorem{observation}{Observation} }

\let\eps\varepsilon
\let\wh\widehat
\let\wt\widetilde
\renewcommand{\d}{{\mathrm d}}
\renewcommand{\Re}{\operatorname{Re}}

\newcommand{\cH}{\mathcal{H}}

\def\pmodModif#1{{\quad \big({\rm mod} \ #1\big)}}

\begin{document}

\newcommand{\arXivNumber}{1805.00544}

\renewcommand{\thefootnote}{}

\renewcommand{\PaperNumber}{086}

\FirstPageHeading

\ShortArticleName{Modularity of Calabi--Yau Manifolds}
\ArticleName{A Hypergeometric Version\\of the Modularity of Rigid Calabi--Yau Manifolds\footnote{This paper is a~contribution to the Special Issue on Modular Forms and String Theory in honor of Noriko Yui. The full collection is available at \href{http://www.emis.de/journals/SIGMA/modular-forms.html}{http://www.emis.de/journals/SIGMA/modular-forms.html}}}

\Author{Wadim ZUDILIN~$^{\dag\ddag\S}$}
\AuthorNameForHeading{W.~Zudilin}

\Address{$^\dag$~Department of Mathematics, IMAPP, Radboud University,\\ \phantom{$^\dag$}~PO Box 9010, 6500~GL Nijmegen, The Netherlands}
\EmailD{\href{mailto:w.zudilin@math.ru.nl}{w.zudilin@math.ru.nl}}
\URLaddressD{\url{http://www.math.ru.nl/~wzudilin/}}

\Address{$^\ddag$~School of Mathematical and Physical Sciences, The University of Newcastle,\\ \phantom{$^\ddag$}~Callaghan, NSW 2308, Australia}
\EmailD{\href{mailto:wadim.zudilin@newcastle.edu.au}{wadim.zudilin@newcastle.edu.au}}

\Address{$^\S$~Laboratory of Mirror Symmetry and Automorphic Forms,\\ \phantom{$^\S$}~National Research University Higher School of Economics,\\ \phantom{$^\S$}~6 Usacheva Str., 119048 Moscow, Russia}
\EmailD{\href{mailto:wzudilin@gmail.com}{wzudilin@gmail.com}}

\ArticleDates{Received May 03, 2018, in final form August 13, 2018; Published online August 17, 2018}

\Abstract{We examine instances of modularity of (rigid) Calabi--Yau manifolds whose periods are expressed in terms of hypergeometric functions. The $p$-th coefficients $a(p)$ of the corresponding modular form can be often read off, at least conjecturally, from the truncated partial sums of the underlying hypergeometric series modulo a power of~$p$ and from Weil's general bounds $|a(p)|\le2p^{(m-1)/2}$, where $m$ is the weight of the form. Furthermore, the critical $L$-values of the modular form are predicted to be $\mathbb Q$-proportional to the values of a~related basis of solutions to the hypergeometric differential equation.}

\Classification{11F33; 11T24; 14G10; 14J32; 14J33; 33C20}

\begin{flushright}
\begin{minipage}{80mm}
\it To Noriko Yui, with wishes to count more points on algebraic varieties rather than years!
\end{minipage}
\end{flushright}

\renewcommand{\thefootnote}{\arabic{footnote}}
\setcounter{footnote}{0}

\section{A prototype}\label{s1}

In \cite{vH97} L.~van Hamme stated some supercongruence analogues of Ramanujan's formulas. The very last observation on van Hamme's list, Conjecture~(M.2) (stated here in an equivalent form),
does not seem to be linked to a known formula though:
\begin{gather}
\sum_{k=0}^{p-1}\frac{(\frac12)_k^4}{k!^4}\equiv a(p) \pmodModif{p^3},\label{e00}
\end{gather}
where $a(n)$ denote the Fourier coefficients of the unique cusp (eigen) form of weight~4 on $\Gamma_0(8)$,
\begin{gather}
f(\tau)=\sum_{n=1}^\infty a(n)q^n=\eta(2\tau)^4\eta(4\tau)^4=q\prod_{m=1}^\infty\big(1-q^{2m}\big)^4\big(1-q^{4m}\big)^4.
\label{e01a}
\end{gather}
Here and below we use the standard hypergeometric notation including $(r)_k=\Gamma(r+k)/\Gamma(r)=\prod\limits_{j=0}^{k-1}(r+j)$ for Pochhammer's symbol; also the congruence $c_1\equiv c_2\pmod{p^\ell}$ for two \emph{rational} numbers is understood as $c_1-c_2\in p^\ell\mathbb Z_p$.
The conjecture \eqref{e00} was later established by T.~Kilbourn in~\cite{Ki06} built on an earlier work
of S.~Ahlgren and K.~Ono in \cite{AO00} on the modularity of the Calabi--Yau threefold $\sum\limits_{j=1}^4\big(x_j+x_j^{-1}\big)=0$.

Interestingly enough, the work of Ahlgren and Ono was motivated by proving a different family of supercongruences for the Ap\'ery numbers
\begin{align*}
A(n)
&=\sum_{k=0}^\infty{\binom nk}^2{\binom{n+k}k}^2
={}_4F_3\biggl(\begin{matrix} -n, \, -n, \, n+1, \, n+1 \\ 1, \, 1, \, 1 \end{matrix}\biggm|1\biggr)
\\
&=\sum_{k=0}^n{\binom nk}^2{\binom{n+k}k}^2 \qquad\text{for}\quad n=0,1,2,\dots
\end{align*}
conjectured by F.~Beukers in \cite{Be87} and established modulo~$p$ there:
\begin{gather}
A\biggl(\frac{p-1}2\biggr)\equiv a(p)\pmodModif{p^2}.
\label{e02a}
\end{gather}
It is not hard to observe that
\begin{gather*}
A\biggl(\frac{p-1}2\biggr)
=\sum_{k=0}^{(p-1)/2}\frac{(\frac{1-p}2)_k^2(\frac{1+p}2)_k^2}{k!^4}
\equiv\sum_{k=0}^{(p-1)/2}\frac{(\frac12)_k^4}{k!^4}
\equiv\sum_{k=0}^{p-1}\frac{(\frac12)_k^4}{k!^4}\pmodModif{p^2},
\end{gather*}
so that \eqref{e02a} follows from \eqref{e00}. On the other hand, the Ap\'ery sequence and the modular paramet\-ri\-zation
of its generating series $\sum\limits_{n=0}^\infty A(n)z^n$ gives one a natural way to construct the right-hand side of~\eqref{e02a}
(namely, the eigenform \eqref{e01a} whose Fourier coefficients show up) modulo~$p$. This construction is performed in~\cite{Be87}
and nicely explained in a certain generality in~\cite{Ve10}. More recently,
V.~Golyshev and D.~Zagier \cite[Section~7]{Za16} show that the $p$-adic interpolation of the coefficients~$a(p)$
of the newform $f(\tau)=\eta(2\tau)^4\eta(4\tau)^4$ is part of a much more general picture that, in particular, predicts that
\begin{gather*}
A(-1/2)={}_4F_3\biggl(\begin{matrix} \frac12, \, \frac12, \, \frac12, \, \frac12 \\ 1, \, 1, \, 1 \end{matrix}\biggm|1\biggr)
=\sum_{k=0}^\infty\frac{(\frac12)_k^4}{k!^4}
\end{gather*}
is rationally proportional to $L(f,2)/\pi^2$, where $L(f,s)$ denotes the $L$-function of the modular form. Furthermore, they prove~\cite{Za16} that
\begin{gather}
{}_4F_3\biggl(\begin{matrix} \frac12, \, \frac12, \, \frac12, \, \frac12 \\ 1, \, 1, \, 1 \end{matrix}\biggm|1\biggr)=\frac{16L(f,2)}{\pi^2},
\label{e03}
\end{gather}
the identity which was independently established in~\cite{RWZ15} via a systematic expressing of critical $L$-values attached to cuspidal $\eta$-products
through hypergeometric functions. Note that the identity~\eqref{e03} is the missing non-$p$-adic counterpart (M.1) of Conjecture~(M.2) from~\cite{vH97};
the latest edition of van Hamme's list can be found in~\cite{Sw15} together with the details about proofs.

One of the principal results in \cite{AO00} is a summation formula for Greene's hypergeometric function, which serves as a finite-field analogue of
the classical hypergeometric series given in~\eqref{e03}.
Curiously enough, R.~Evans in his review \cite{Ev00} of \cite{AO00} mentions that no summation formula is known for this $_4F_3$-value in~\eqref{e03};
the evaluation \eqref{e03} established in \cite{RWZ15,Za16} thus fills in this gap in the hypergeometric literature.

A principal goal of this note is to put the pair \eqref{e00}, \eqref{e03} in a broader context of relationship between classical generalized hypergeometric functions and the $L$-values of modular forms. This is performed here more in the spirit of Golyshev's gamma structures \cite{GM14} rather than hypergeometric motives \cite{RRV17,RV17} of F.~Rodriguez Villegas and others. At the same time, we do not pretend to be too broad in our exposition, mainly highlighting certain specific arithmetic and analytical perspectives which we find aesthetically appealing.

\section{Modularity of Calabi--Yau threefolds}\label{s2}

The Calabi--Yau threefold in Section~\ref{s1} comes as a part of the complete intersection of four degree 2 surfaces in $\mathbb P^8$; the periods of the latter family of threefolds satisfy the hypergeometric equation whose unique analytical solution is
\begin{gather*}
{}_4F_3\biggl(\begin{matrix} \frac12, \, \frac12, \, \frac12, \, \frac12 \\ 1, \, 1, \, 1 \end{matrix}\biggm|z\biggr)
=\sum_{k=0}^\infty\frac{(\frac12)_k^4}{k!^4} z^k.
\end{gather*}
Namely, the fiber $z=1$ corresponds to the rigid Calabi--Yau threefold $\sum\limits_{j=1}^4\big(x_j+x_j^{-1}\big)=0$.

There are fourteen `hypergeometric' families of Calabi--Yau threefolds whose periods are solutions of hypergeometric equations with parameters $(r,1-r,t,1-t)$, where
\begin{gather*}
(r,t) =\big(\tfrac12,\tfrac12\big), \;
\big(\tfrac12,\tfrac13\big), \;
\big(\tfrac12,\tfrac14\big), \;
\big(\tfrac12,\tfrac16\big), \;
\big(\tfrac13,\tfrac13\big), \;
\big(\tfrac13,\tfrac14\big), \;
\big(\tfrac13,\tfrac16\big), \\
\hphantom{(r,t) =}{} \
\big(\tfrac14,\tfrac14\big), \;
\big(\tfrac14,\tfrac16\big), \;
\big(\tfrac16,\tfrac16\big), \;
\big(\tfrac15,\tfrac25\big), \;
\big(\tfrac18,\tfrac38\big), \;
\big(\tfrac1{10},\tfrac3{10}\big), \;
\big(\tfrac1{12},\tfrac5{12}\big),
\end{gather*}
and the modularity from Section~\ref{s1} is expected to be extendable to all families as follows.

\begin{observation}\label{obs1}Let a pair $(r,t)$ be from the list. For a prime $p$ not dividing the denominators of $r$ and $t$, define $a(p)$ to be the smallest (in absolute value) integer residue modulo $p^3$ of the partial sum
\begin{gather*}
\sum_{k=0}^{p-1}\frac{(r)_k(1-r)_k(t)_k(1-t)_k}{k!^4}
\end{gather*}
of the hypergeometric series
\begin{gather*}
{}_4F_3\biggl(\begin{matrix} r, \, 1-r, \, t, \, 1-t \\ 1, \, 1, \, 1 \end{matrix}\biggm|1\biggr)
=\sum_{k=0}^\infty\frac{(r)_k(1-r)_k(t)_k(1-t)_k}{k!^4}.
\end{gather*}
Then $|a(p)|\le2p^{3/2}$ and $a(p)$ are the Fourier coefficients of a suitable eigenform $f(\tau)=q+a(2)q^2+\dotsb$ of weight~4
for some congruence subgroup of ${\rm PSL}_2(\mathbb Z)$.

Furthermore, introduce a special (normalized Frobenius) basis of solutions of the differential equation for
\begin{gather*}
F_0(z)={}_4F_3\biggl(\begin{matrix} r, \, 1-r, \, t, \, 1-t \\ 1, \, 1, \, 1 \end{matrix}\biggm|z\biggr)
\end{gather*}
as the first coefficients in the Taylor $\eps$-expansion of the (bilateral) hypergeometric function
\begin{gather}
\frac1{\Gamma(r) \Gamma(1-r) \Gamma(t) \Gamma(1-t)}\nonumber\\
\qquad\quad{} \times \sum_{n=-\infty}^\infty\frac{\Gamma(r+\eps+n) \Gamma(1-r+\eps+n) \Gamma(t+\eps+n) \Gamma(1-t+\eps+n)}{\Gamma(1+\eps+n)^4} z^{n+\eps} \nonumber\displaybreak[2]\\
\qquad{} =\frac1{\Gamma(r)\Gamma(1-r)\Gamma(t)\Gamma(1-t)} \nonumber\\
\qquad\quad{} \times \sum_{n=0}^\infty\frac{\Gamma(r+\eps+n)\Gamma(1-r+\eps+n)\Gamma(t+\eps+n)\Gamma(1-t+\eps+n)} {\Gamma(1+\eps+n)^4}z^{n+\eps}+O\big(\eps^4\big)\nonumber\displaybreak[2]\\
\qquad{} =F_0(z)+F_1(z)\eps+F_2(z)\eps^2+F_3(z)\eps^3+O\big(\eps^4\big)\qquad\text{as}\quad \eps\to0.\label{eq:comp}
\end{gather}
Then numerical calculations suggest conjectural inclusions
\begin{gather}
\frac{L(f,1)}{F_1(1)}\in\mathbb Q, \qquad
\frac{L(f,2)}{F_2(1)}\in\mathbb Q
\qquad\text{and}\qquad
\frac{L(f,3)}{F_3(1)}\in\mathbb Q.\label{LF4}
\end{gather}
\end{observation}

\begin{remark}\label{rem0}Observation \ref{obs1} contains an explicit algorithm for reconstructing the Hecke eigenvalues $a(p)$, so it is straighforward to compute them numerically for good primes $p$ from the partial sums. This supercongruence part has been already exploited by F.~Rodriguez Villegas in~\cite{RV03} who noticed that the truncated hypergeometric sums are congruent to $a(p)$ modulo~$p^3$ and used this fact to identify the corresponding eigenforms $f(\tau)$ and their levels. The knowledge of Hecke eigenvalues $a(p)$ allows one to reconstruct all Fourier coefficients of $f(\tau)=\sum\limits_{n=1}^\infty a(n)q^n$ from the Euler product of the $L$-function $L(f,s)=\sum\limits_{n=1}^\infty a(n)n^{-s}$. Missing finitely many $a(p)$ in the Euler product has no effect on the inclusions~\eqref{LF4}.
\end{remark}

\begin{table}[h]\centering
\caption{Eigenforms for rigid Calabi--Yau manifolds.}\label{tab1}
\vspace{1mm}

\begin{tabular}{|c|c|l|l|}
\hline
$(r,t)$ & $f(\tau)$ & \quad\;\ level & LMFDB label \cite{LMFDB} \\
\hline
$\big(\tfrac12,\tfrac12\big)$ & $\eta_2^4\eta_4^4$ & $\phantom{00}8=2^3\vphantom{|^{0^0}}$ & \qquad\phantom{00}\href{http://www.lmfdb.org/ModularForm/GL2/Q/holomorphic/8/4/1/a/}{8.4.1.a} \\[.5mm]
$\big(\tfrac12,\tfrac13\big)$ & $\eta_6^{14}/\big(\eta_2^3\eta_{18}^3\big)-3\eta_2^3\eta_6^2\eta_{18}^3$ & $\phantom036=2^2\cdot3^2$ & \qquad\phantom{0}\href{http://www.lmfdb.org/ModularForm/GL2/Q/holomorphic/36/4/1/a/}{36.4.1.a} \\[.5mm]
$\big(\tfrac12,\tfrac14\big)$ & $\eta_4^{16}/\big(\eta_2^4\eta_8^4\big)$ & $\phantom016=2^4$ & \qquad\phantom{0}\href{http://www.lmfdb.org/ModularForm/GL2/Q/holomorphic/16/4/1/a/}{16.4.1.a} \\[.5mm]
$\big(\tfrac12,\tfrac16\big)$ & & $\phantom072=2^3\cdot3^2$ & \qquad\phantom{0}\href{http://www.lmfdb.org/ModularForm/GL2/Q/holomorphic/72/4/1/b/}{72.4.1.b} \\[.5mm]
$\big(\tfrac13,\tfrac13\big)$ & $\eta_1^3\eta_3^4\eta_9-27\eta_3\eta_9^4\eta_{27}^3$ & $\phantom027=3^3$ & \qquad\phantom{0}\href{http://www.lmfdb.org/ModularForm/GL2/Q/holomorphic/27/4/1/a/}{27.4.1.a} \\[.5mm]
$\big(\tfrac13,\tfrac14\big)$ & $\eta_3^8$ & $\phantom{00}9=3^2$ & \qquad\phantom{00}\href{http://www.lmfdb.org/ModularForm/GL2/Q/holomorphic/9/4/1/a/}{9.4.1.a} \\[.5mm]
$\big(\tfrac13,\tfrac16\big)$ & & $108=2^2\cdot3^3$ & \qquad\href{http://www.lmfdb.org/ModularForm/GL2/Q/holomorphic/108/4/1/a/}{108.4.1.a} \\[.5mm]
$\big(\tfrac14,\tfrac14\big)$ & $\eta_4^{10}/\eta_8^2-8\eta_8^{10}/\eta_4^2$ & $\phantom032=2^5$ & \qquad\phantom{0}\href{http://www.lmfdb.org/ModularForm/GL2/Q/holomorphic/32/4/1/a/}{32.4.1.a} \\[.5mm]
$\big(\tfrac14,\tfrac16\big)$ & $\eta_{12}^{32}/\big(\eta_6^{12}\eta_{24}^{12}\big)+16\eta_6^4\eta_{24}^4$ & $144=2^4\cdot3^2$ & \\[.5mm]
$\big(\tfrac16,\tfrac16\big)$ & & $216=2^3\cdot3^3$ & \\[.5mm]
$\big(\tfrac15,\tfrac25\big)$ & $\eta_5^{10}/(\eta_1\eta_{25})+5\eta_1^2\eta_5^4\eta_{25}^2$ & $\phantom025=5^2$ & \qquad\phantom{0}\href{http://www.lmfdb.org/ModularForm/GL2/Q/holomorphic/25/4/1/b/}{25.4.1.b} \\[.5mm]
$\big(\tfrac18,\tfrac38\big)$ & & $128=2^7$ & \\[.5mm]
$\big(\tfrac1{10},\tfrac3{10}\big)$ & & $200=2^3\cdot5^2$ & \\[.5mm]
$\big(\tfrac1{12},\tfrac5{12}\big)$ & & $864=2^5\cdot3^3$ & \\[.5mm] \hline
\end{tabular}
\end{table}

\begin{remark}\label{rem1} The prediction about the relationship between the critical $L$-values and the hyper\-geo\-metric values $F_1(1)$, $F_2(1)$, $F_3(1)$ is due to V.~Golyshev, and it is a part of general phenomenon. The fact that the coefficients $F_j(z)$ are solutions of the hypergeometric equation for~$F_0(z)$ is established in \cite[Section~3]{GM14}; we survey some information about this from a `hypergeometric' perspective in Section~\ref{bilateral}. None of the relations in~\eqref{LF4} seem to be proved.

Accidentally, when $r=t=\frac12$, we have an extra rational relation $F_0(1)=F_2(1)/\big(2\pi^2\big)$, and it is this equality that originates the anticipated equality \eqref{e03} (rigorously established!). It is the only case when $F_0(1)$ is linearly dependent over~$\mathbb Q$ with $F_j(1)/\pi^j$ for $j=1,2,3$.
\end{remark}

\begin{remark}\label{rem1a} The case $(r,t)=\big(\frac13,\frac14\big)$ in Observation~\ref{obs1} corresponds to a particularly simple CM modular form of level~9, namely, to $f(\tau)=\eta(3\tau)^8$. Its critical $L$-values possess closed-form evaluation
\begin{gather*}
L\big(\eta(3\tau)^8,2\big)=\frac{\Gamma(1/3)^9}{96\pi^4} \qquad\text{and}\qquad L\big(\eta(3\tau)^8,3\big)=\frac{\Gamma(1/3)^9}{144\sqrt3 \pi^3}
\end{gather*}
(the strategy for this computation is set up in Damerell's work \cite{Da70}). All 14 cases correspond to rigid Calabi--Yau threefolds defined over $\mathbb Q$ and hence they do correspond to modular forms of weight~4 for some congruence subgroups of ${\rm PSL}_2(\mathbb Z)$. Table~\ref{tab1} records the instances of modular forms, for which we know their eta-product expressions; the notation~$\eta_m$ stands for~$\eta(m\tau)$.
\end{remark}

\begin{remark} \label{rem2} The eigenform $f(\tau)$ in Observation~\ref{obs1}, namely, the eigenvalues $a(p)$, are related to the counting of points modulo~$p$ on the (rigid) Calabi--Yau threefold corresponding to $z=1$ in the family. This counting naturally leads to representations of $a(p)$ by means of finite-field hypergeometric functions~-- due to J.~Greene \cite{Gr87}, D.~McCarthy \cite{Mc12} and, in a greater generality, F.~Beukers, H.~Cohen, A.~Mellit~\cite{BCM15}~-- the representations that are used in the proof of Observation~\ref{obs1} in the case $r=t=\frac12$. All 14 cases in the observation, namely the modulo~$p^3$ supercongruences, are now proved simultaneously and rigorously in the joint paper \cite{LTYZ17} with L.~Long, F.-T.~Tu and N.~Yui.
\end{remark}

\section{Bilateral hypergeometric functions and hypertrigonometry}\label{bilateral}

In this section we will examine the bilateral hypergeometric sum
\begin{gather}
{}_m\cH_m\biggl(\begin{matrix} a_1, \, \dots, \, a_m \\ b_1, \, \dots, \, b_m \end{matrix}\biggm|z;\eps\biggr)
=\frac{\prod\limits_{j=1}^m\Gamma(b_j)}{\prod\limits_{j=1}^m\Gamma(a_j)}
 \sum_{n=-\infty}^\infty\frac{\prod\limits_{j=1}^m\Gamma(a_j+\eps+n)}{\prod\limits_{j=1}^m\Gamma(b_j+\eps+n)} z^{n+\eps}\label{beq1}
\end{gather}
from both classical \cite[Chapter~6]{Sl66} and recent \cite{GM14} perspectives. For fixed $\eps\in\mathbb C$ (different from the poles
of the gamma functions $\Gamma(a_j+\eps+n)$) and a generic set of complex parameters~$a_j$ and~$b_j$, $j=1,\dots,m$, satisfying
\begin{gather*}
\Re(b_1+\dots+b_m)>\Re(a_1+\dots+a_m)
\end{gather*}
the defining series converges on the unit circle $|z|=1$. Our principal interest will be in the case $b_1=\dots=b_m=1$.
On using
\begin{gather*}
\left(z\frac{\d}{\d z}+a\right)z^{n+\eps}=(a+\eps+n)z^{n+\eps}
\end{gather*}
and the basic property of the gamma function we arrive at the following.

\begin{lemma}\label{blem1}The function \eqref{beq1} satisfies the $($linear differential$)$ hypergeometric equation
\begin{gather}
\left(z\prod_{j=1}^m\left(z\frac{\d}{\d z}+a_j\right)-\prod_{j=1}^m\left(z\frac{\d}{\d z}+b_j-1\right)\right){}_m\cH_m(z;\eps)=0\label{beq2}
\end{gather}
on the circle $|z|=1$.
\end{lemma}

The function \eqref{beq1} can be analytically continued from the unit circle to the $\mathbb C$-plane with cuts along the real intervals $(-\infty,0]$ and $[1,+\infty)$
by relating it to the bilateral hypergeometric function \cite[equation~(6.1.2.3)]{Sl66},
\begin{align*}
{}_mH_m\biggl(\begin{matrix} a_1, \, \dots, \, a_m \\ b_1, \, \dots, \, b_m \end{matrix}\biggm|z\biggr)
&=\sum_{n=-\infty}^\infty\frac{(a_1)_n\dotsb(a_m)_n}{(b_1)_n\dotsb(b_m)_n}cz^n
\\ &
={}_{m+1}F_m\biggl(\begin{matrix} 1, \, a_1, \, \dots, \, a_m \\ b_1, \, \dots, \, b_m \end{matrix}\biggm|z\biggr)
\\ &\quad
+\frac{(b_1-1)\dotsb(b_m-1)}{(a_1-1)\dotsb(a_m-1)}\,
{}_{m+1}F_m\biggl(\begin{matrix} 1, \, 2-b_1, \, \dots, \, 2-b_m \\ 2-a_1, \, \dots, \, 2-a_m \end{matrix}\biggm|\frac1z\biggr),
\end{align*}
where the (extended to negative) Pochhammer symbol is
\begin{gather*}
(a)_n=\frac{\Gamma(a+n)}{\Gamma(a)}
=\begin{cases}
1 &\text{if} \ n=0, \\
a(a+1)\dotsb(a+n-1) &\text{if} \ n>0, \\[1mm]
\dfrac{1}{(a-1)(a-2)\dotsb(a-(-n))} &\text{if} \ n<0.
\end{cases}
\end{gather*}

\begin{lemma}[see also \cite{GM14}]\label{blem2} As function of $z$, the function \eqref{beq1} is continued analytically to $\mathbb C\setminus(-\infty,0]\cup[1,+\infty)$ by means of the hypergeometric functions
as follows:
\begin{align*}
&
{}_m\cH_m\biggl(\begin{matrix} a_1, \, \dots, \, a_m \\ b_1, \, \dots, \, b_m \end{matrix}\biggm|z;\eps\biggr)
=\frac{z^\eps\prod\limits_{j=1}^m\Gamma(a_j+\eps)\,\Gamma(b_j)}{\prod\limits_{j=1}^m\Gamma(a_j)\,\Gamma(b_j+\eps)}
\biggl\{{}_{m+1}F_m\biggl(\begin{matrix} 1, \, a_1+\eps, \, \dots, \, a_m+\eps \\ b_1+\eps, \, \dots, \, b_m+\eps \end{matrix}\biggm|z\biggr)
\\ &\qquad
+\prod_{j=1}^m\frac{b_j+\eps-1}{a_j+\eps-1}\,
{}_{m+1}F_m\biggl(\begin{matrix} 1, \, 2-b_1-\eps, \, \dots, \, 2-b_m-\eps \\ 2-a_1-\eps, \, \dots, \, 2-a_m-\eps \end{matrix}\biggm|z^{-1}\biggr)\biggr\},
\end{align*}
and the analytic continuation satisfies the hypergeometric equation \eqref{beq2}.
\end{lemma}

In particular, the lemma implies that
\begin{align*}
&\frac1{\prod\limits_{j=1}^m\Gamma(a_j)} \sum_{n=-\infty}^\infty\frac{\prod\limits_{j=1}^m\Gamma(a_j+\eps+n)}{\Gamma(1+\eps+n)^m} z^{n+\eps}
\\ &\qquad
=\frac{z^\eps\prod\limits_{j=1}^m\Gamma(a_j+\eps)}{\Gamma(1+\eps)^m\prod\limits_{j=1}^m\Gamma(a_j)}\,
{}_{m+1}F_m\biggl(\begin{matrix} 1, \, a_1+\eps, \, \dots, \, a_m+\eps \\ 1+\eps, \, \dots, \, 1+\eps \end{matrix}\biggm|z\biggr)
+O\big(\eps^m\big),
\end{align*}
the reduction we used in computation \eqref{eq:comp} of Section~\ref{s2}.

Finally, notice that the sum in \eqref{beq1} is invariant under the shifts of $\eps$ by integers, and the principal result of \cite{GM14} can be stated as follows.

\begin{lemma}\label{blem3}As function of $\eps$, the function \eqref{beq1} is periodic with period~$1$. Furthermore, its normalization
\begin{gather}
\prod_{j=1}^m\sin\pi(a_j+\eps)\times{}_m\cH_m\biggl(\begin{matrix} a_1, \, \dots, \, a_m \\ b_1, \, \dots, \, b_m \end{matrix}\biggm|z;\eps\biggr)\label{bL3}
\end{gather}
is a $\mathbb C$-linear combination of $e^{\pi ik\eps}$, where $|k|\le m$ and $k\equiv m\pmod2$. This means that the Fourier expansion of the latter function is a finite Fourier polynomial, whose coefficients depend only on~$z$.
\end{lemma}

\begin{proof}Using the reflection property of the gamma function we find
\begin{align*}
\Gamma(a+\eps+n)=\frac{\pi}{\sin\pi(a+\eps+n)} \frac1{\Gamma(1-a-\eps-n)}=\frac{\pi}{\sin\pi(a+\eps)} \frac{(-1)^n}{\Gamma(1-a-\eps-n)},
\end{align*}
so that
\begin{align*}
{}_m\cH_m\biggl(\begin{matrix} a_1, \, \dots, \, a_m \\ b_1, \, \dots, \, b_m \end{matrix}\biggm|z;\eps\biggr)
&=\frac{z^\eps \pi^m\prod\limits_{j=1}^m\Gamma(b_j)}{\prod\limits_{j=1}^m\Gamma(a_j) \sin\pi(a_j+\eps)}
\\ &\quad\times \sum_{n=-\infty}^\infty\frac{(-1)^{mn}z^n}{\prod\limits_{j=1}^m\Gamma(1-a_j-\eps-n) \Gamma(b_j+\eps+n)}.
\end{align*}
It remains to notice that the functions
\begin{gather*}
\frac1{\prod\limits_{j=1}^m\Gamma(1-a_j-\eps-n) \Gamma(b_j+\eps+n)}
\end{gather*}
are entire and estimate their growth as $\eps\to\infty$ (see \cite[Theorem 1.5]{GM14}).
\end{proof}

\begin{remark}\label{remB}Though Lemma~\ref{blem3} (and the estimates from \cite{GM14}) guarantee that at most $m+1$ terms show up in the Fourier expansion of \eqref{bL3}, in reality one does not get the term $e^{-\pi im\eps}$ (or $e^{\pi im\eps}$) when $\Re z>0$ (or $\Re z<0$, respectively). In the case when $z$ is real from the interval $0<z<1$, we still need to specify along which bank of the real line we proceed; for convenience, from now on we agree to use the upper bank.

Even more, if $z=1$ and the corresponding bilateral hypergeometric series converge at this special point then the both terms $e^{-\pi im\eps}$ and $e^{\pi im\eps}$ in the Fourier expansion of \eqref{bL3} do not show up. This allows one to rigorously establish that $F_1(1)$ and $F_3(1)/\pi^2$ are rationally proportional~-- something that could follow from~\eqref{LF4} complemented with the Manin--Shimura relation of the critical $L$-values~\cite{Sh76,Sh77}.
\end{remark}

\section{A hypergeometric modularity of elliptic curves}\label{s3}

Probably, the most classical version of the observation above refers to the modularity of elliptic curves (that is, Calabi--Yau onefolds). Our principal illustration will deal with the family
\begin{gather*}
E_z\colon \ y^2=x(1-x)(x-z), \qquad z\in\mathbb C\setminus\{0,1,\infty\},
\end{gather*}
which is a twist of the classical Legendre family of elliptic curves
\begin{gather*}
\wh E_z\colon \ y^2=x(x-1)(x-z), \qquad z\in\mathbb C\setminus\{0,1,\infty\}.
\end{gather*}
In fact, performing the change $x\mapsto1-x$ we see that the curves $E_{1-z}$ and $\wh E_z$ are isomorphic.

Let $p$ be an odd prime and $z\in\mathbb Q$ be $p$-integral not equal to 0 or 1. By Hasse's theorem \cite[Theorem~V.1.1]{Si86} the number of points on the curve $\wh E_z/\mathbb F_p$ satisfies
\begin{gather*}
\big|\#\big(\wh E_z/\mathbb F_p\big)-(p+1)\big|\le2\sqrt p.
\end{gather*}
On the other hand, it follows from the proof of Theorem~V.4.1(b) in \cite{Si86} that
\begin{align*}
\#\big(\wh E_z/\mathbb F_p\big)-1&\equiv(-1)^{(p-1)/2}\sum_{k=0}^{(p-1)/2}{\binom{(p-1)/2}k}^2z^k\pmod p \\
&\equiv(-1)^{(p-1)/2}\sum_{k=0}^{(p-1)/2}\frac{(\frac12)_k^2}{k!^2} z^k \equiv(-1)^{(p-1)/2}\sum_{k=0}^{p-1}\frac{(\frac12)_k^2}{k!^2} z^k\pmod p.
\end{align*}

By combining the two results above we conclude that the integer $\wh a(p)=\wh a(p;z)=\#(\wh E_z/\mathbb F_p)-(p+1)$ satisfies Weil's bound $|\wh a(p)|\le2\sqrt p$ and the congruence
\begin{gather*}
\wh a(p)\equiv\left(\frac{-4}p\right)\sum_{k=0}^{p-1}\frac{(\frac12)_k^2}{k!^2} z^k\pmod p,
\end{gather*}
where $\bigl(\frac{-4}{\cdot}\bigr)$ denotes the quadratic character modulo~4. By the modularity theorem the numbers~$\wh a(p)$ build up to the $L$-function of the elliptic curve~$\wh E_z$,
\begin{gather*}
L\big(\wh E_z,s\big)=\prod_{p}\big(1-\wh a(p)p^{-s}+\varepsilon_pp^{1-2s}\big)^{-1} =\sum_{n=1}^\infty\frac{\wh a(n)}{n^s}, \qquad \varepsilon_p\in\{0,1\}.
\end{gather*}
Furthermore, the central (critical) value of $L\big(\wh E_z,s\big)$ is rationally proportional to a period of the curve $\wh E_z$, namely, to the period
\begin{gather*}
\Re\int_1^\infty\frac{\d x}{\sqrt{x(x-1)(x-z)}}=\Re\int_0^1\frac{\d t}{\sqrt{t(1-t)(1-zt)}}
=\pi\,\Re{}_2F_1\biggl(\begin{matrix} \frac12, \, \frac12 \\ 1 \end{matrix}\biggm|z\biggr),
\end{gather*}
where we made the change of variable $x=1/t$ in the former integral. The real part can be omitted when $z<1$.

In order to state the above for the family of elliptic curves $E_z\simeq\wh E_{1-z}$ we notice first that the above calculation of the Hasse invariant from~\cite{Si86} implies the congruence
\begin{gather}
\biggl(\frac{-4}p\biggr)\sum_{k=0}^{p-1}\frac{(\frac12)_k^2}{k!^2}z^k\equiv\sum_{k=0}^{p-1}\frac{(\frac12)_k^2}{k!^2}(1-z)^k\pmod p.\label{a1-z}
\end{gather}
Second, writing for a real $r$ in the range $0<r<1$,
\begin{align}
F(z;\eps)&=\frac1{\Gamma(r)\,\Gamma(1-r)}\sum_{n=-\infty}^\infty\frac{\Gamma(r+\eps+n)\,\Gamma(1-r+\eps+n)}{\Gamma(1+\eps+n)^2}z^{n+\eps}\nonumber\\
&=\frac1{\Gamma(r)\Gamma(1-r)}\sum_{n=0}^\infty\frac{\Gamma(r+\eps+n)\Gamma(1-r+\eps+n)}{\Gamma(1+\eps+n)^2}z^{n+\eps}+O\big(\eps^2\big)\nonumber\\
&=F_0(z)+F_1(z)\eps+O\big(\eps^2\big)\qquad\text{as}\quad \eps\to0,\label{F01}
\end{align}
where
\begin{gather*}
F_0(z)={}_2F_1\biggl(\begin{matrix} r, \, 1-r \\ 1 \end{matrix}\biggm|z\biggr),
\end{gather*}
and applying the monodromy of the hypergeometric function we obtain
\begin{gather*}
F_1(z)=-\Gamma(r)\Gamma(1-r)F_0(1-z)=-\frac\pi{\sin\pi r} F_0(1-z).
\end{gather*}
This relation is valid in the cut $\mathbb C$-plane $\mathbb C\setminus(-\infty,0]\cup[1,\infty)$ but also along the respective banks of the cuts; in particular, for the real parts, the identity
\begin{gather}
\Re F_1(z)=-\frac\pi{\sin\pi r} \Re F_0(1-z) \label{F1-z}
\end{gather}
is true for any complex $z\ne0,1$. Using \eqref{F1-z} with $r=\frac12$ we can summarize our findings as follows.

\begin{observation}\label{obs2}Let $p>2$ be a prime not dividing the denominator of a given $z\in\mathbb Q\setminus\{0,1\}$. Define the integer $a(p)=a(p;z)$ as the absolutely smallest residue modulo $p$ of the partial sum
\begin{gather*}
\sum_{k=0}^{p-1}\frac{(\frac12)_k^2}{k!^2} z^k
\end{gather*}
(so that $-p/2<a(p)<p/2$) of the hypergeometric function
\begin{gather*}
F_0(z)={}_2F_1\biggl(\begin{matrix} \frac12, \, \frac12 \\ 1 \end{matrix}\biggm|z\biggr)=\sum_{k=0}^\infty\frac{(\frac12)_k^2}{k!^2} z^k.
\end{gather*}
Then the number satisfies Weil's estimate $|a(p)|<2\sqrt p$.

Furthermore, form the associated $L$-function
\begin{gather*}
L(z,s)=\prod_{p}\big(1-a(p)p^{-s}+p^{1-2s}\big)^{-1}=\sum_{n=1}^\infty\frac{a(n)}{n^s},
\end{gather*}
where the product is over primes $p>2$ that do not divide the denominator of $z$. Then
\begin{gather}
\frac{L(z,1)}{\Re F_1(z)}=-\frac{L(z,1)}{\pi\,\Re F_0(1-z)}\in\mathbb Q,\label{L=F}
\end{gather}
where $F_1(z)$ originates from the $\eps$-expansion \eqref{F01}.
\end{observation}

Note that $a(p)$ constructed in Observation~\ref{obs2} may in fact differ, by a multiple of $p$, from the $p$-th Fourier coefficient of the modular form associated with $E_z$ for the range $p\le13$. However the change (or omission) of finite set of factors in the product defining $L(E_z,s)$ contributes by a nonzero \emph{rational} factor in $L(z,1)$, so that relation \eqref{L=F} is seen to be equivalent to
\begin{gather*}
\frac{L(E_z,1)}{\Re F_1(z)}
=-\frac{L(E_z,1)}{\pi\,\Re F_0(1-z)}
\in\mathbb Q.
\end{gather*}
We also stress on the fact that $L(E_z,1)$, therefore $L(z,1)$ in \eqref{L=F}, vanishes when the (analytic) rank of the elliptic curve $E_z$ is positive. In such situations, numerics suggests no relation between the hypergeometric functions $F_0(z)$, $F_1(z)$ in question and the first nonzero derivative of $L(E_z,s)$ (or of $L(z,s)$) at $s=1$.

A similar analysis applies to three other classical hypergeometric series
\begin{gather}
{}_2F_1\biggl(\begin{matrix} r, \, 1-r \\ 1 \end{matrix}\biggm|z\biggr)
=\sum_{k=0}^\infty\frac{(r)_k(1-r)_k}{k!^2} z^k,
\qquad \text{where}\quad r\in\big\{\tfrac13,\tfrac14,\tfrac16\big\}.\label{2F1}
\end{gather}
They are known to represent the periods of suitable families of elliptic curves, for example, of the pencils of elliptic curves
\begin{gather*}
X^2Y+Y^2Z+Z^2X=z^{1/3}XYZ, \\ X^4+Y^2+Z^4=z^{1/4}XYZ \qquad\text{and}\qquad X^3+Y^2+Z^6=z^{1/6}XYZ,
\end{gather*}
respectively, in weighted projective planes~\cite{St06}. The corresponding Weierstrass forms are
\begin{gather*}
y^2=x^3-3(9-8z)x+2\big(27-36z+8z^2\big), \\
y^2=x^3-27(1+3z)x+54(1-9z) \qquad\text{and}\qquad y^2=x^3-27x+54(1-2z).
\end{gather*}

\begin{observation} \label{obs3}Take $r\in\big\{\frac13,\frac14,\frac16\big\}$ and $z\in\mathbb Q\setminus\{0,1\}$. Let $p$ be a prime not dividing the denominators of~$r$ and~$z$. Define the integer $a(p)=a(p;r,z)$ as the absolutely smallest residue modulo $p$ of the partial sum
\begin{gather*}
\sum_{k=0}^{p-1}\frac{(r)_k(1-r)_k}{k!^2} z^k
\end{gather*}
of the hypergeometric function \eqref{2F1}. Then the number satisfies Weil's estimate $|a(p)|<2\sqrt p$.

Form the associated $L$-function
\begin{gather*}
L(z,s)=\prod_{p}\big(1-a(p)p^{-s}+p^{1-2s}\big)^{-1}=\sum_{n=1}^\infty\frac{a(n)}{n^s},
\end{gather*}
where the product is over primes $p$ that do not divide the denominators of $r$ and $z$. Then
\begin{gather*}
\frac{L(z,1)}{\Re F_1(z)} =-\frac{L(z,1)}{\Gamma(r)\Gamma(1-r)\Re F_0(1-z)}\in\mathbb Q,
\end{gather*}
where $F_1(z)$ originates from the corresponding $\eps$-expansion \eqref{F01}.
\end{observation}

We can also point out the symmetry property $a(p;r,z)=\chi(p)a(p;r,1-z)$ valid for $r\in\big\{\frac12,\frac13,\frac14,\frac16\big\}$ (see \eqref{a1-z} for $r=\frac12$) and all admissible primes $p$ with the corresponding choice of the quadratic character
\begin{gather*}
\chi(\,\cdot\,)=\left(\frac{-4}{\cdot}\right), \; \left(\frac{-3}{\cdot}\right), \; \left(\frac{-2}{\cdot}\right) \;\text{or}\; \left(\frac{-4}{\cdot}\right)
\qquad\text{for}\quad
r=\frac12, \; \frac13, \; \frac14, \; \frac16, \; \text{respectively}.
\end{gather*}

\begin{remark}\label{rem3} With each modular form $f(\tau)$ of integral weight at least 2 one can canonically associate two periods $\omega_-$ and $\omega_+$. When the weight higher than 2 shows up, and these are examples from Section~\ref{s2} above and Section~\ref{s4} below, the critical $L$-values $L(f,m)/\pi^m$ represent the both periods $\omega_-$ and $\omega_+$ of the modular form, so that twisting the Hecke eigenvalues $a(p)$ by an odd character is equivalent to changing the parity of~$m$ or swapping the periods. This is an immediate consequence of the Manin--Shimura description of the critical $L$-values \cite{Sh76,Sh77}. In situations covered in this section the modular forms $f(\tau)$ have weight~2; thus, the symmetry $a(p;r,z)=\chi(p)a(p;r,1-z)$ under the involution $z\mapsto1-z$ displays the interchange of the periods $\omega_-$ and $\omega_+$ on the corresponding elliptic curve in the family.
\end{remark}

The potentials of the hypergeometric description of the modularity are at least two-fold. First, they provide us with a new class of summation theorems for arithmetic instances of classical Euler--Gauss hypergeometric function (cf.~\cite{Ze92}). Second, they allow one to deal with elliptic curves defined over algebraic extensions of $\mathbb Q$ as the hypergeometric machinery works for not necessarily rational~$z$, at least formally.

\section{Other modularity instances}\label{s4}

One interesting message coming from Observation~\ref{obs1} is that $z=1$ always corresponds to a~rigid Calabi--Yau threefold in each hypergeometric family. Note that $z=1$ happens to be a singular point of the related hypergeometric differential equation, so an expectation is that Observation~\ref{obs1} can be suitably extended to some non-hypergeometric families and the Calabi--Yau manifolds corresponding to some singularities of the underlying Picard--Fuchs differential equations. But rigid Calabi--Yau manifolds can correspond to non-singular points $z$ as the observations in Section~\ref{s3} demonstrate. We can also record vaguely the following observation about potential instances of the modularity of Calabi--Yau twofolds (that is, $K3$ surfaces with Picard rank~20), where some non-singular points show up.

\begin{observation}\label{obs4}Let $r\in\big\{\frac12,\frac13,\frac14,\frac16\big\}$ and let rational $z$ be 1 or `arithmetically special' (that is, corresponding to CM cases of the underlying modular parametrization~-- we address this point more specifically in Remark~\ref{clausen}). For a prime $p$ not dividing the denominators of $r$ and $z$, define~$a(p)$ to be the absolute smallest integer residue modulo~$p^2$ of the partial sum
\begin{gather*}
\sum_{k=0}^{p-1}\frac{(\frac12)_k(r)_k(1-r)_k}{k!^3} z^k
\end{gather*}
of the hypergeometric series
\begin{gather*}
{}_3F_2\biggl(\begin{matrix} \frac12, \, r, \, 1-r \\ 1, \, 1 \end{matrix}\biggm|z\biggr)
=\sum_{k=0}^\infty\frac{(\frac12)_k(r)_k(1-r)_k}{k!^3} z^k.
\end{gather*}
Then $|a(p)|\le2p$ and $a(p)$ are the Fourier coefficients of a suitable eigenform $f(\tau)=q+a(2)q^2+\dotsb$ of weight~3 for some congruence subgroup of ${\rm PSL}_2(\mathbb Z)$. Furthermore, in several cases we have
\begin{gather*}
\frac{L(f,2)}{\pi^2}\in\mathbb Q\big[\sqrt d\big] \Re{}_3F_2\biggl(\begin{matrix} \frac12, \, r, \, 1-r \\ 1, \, 1 \end{matrix}\biggm|z\biggr)
\end{gather*}
and then also a similar inclusion for $L(f,1)/\pi$. Here $d\in\mathbb Z$ depend on the data $r$, $z$ and on the choice of $m$ in $L(f,m)/\pi^m$.
\end{observation}

The following illustrations all correspond to the choice $r=\frac12$ and are motivated by the results established in \cite{SB85}. The corresponding character $\chi$ is trivial and we have
\begin{gather*}
\sum_{k=0}^{p-1}\frac{(\frac12)_k^3}{k!^3}
\equiv a_1(p)\pmodModif{p^2}
=\begin{cases}
2\big(a^2-b^2\big) & \text{if $p=a^2+b^2$, $a$ odd}, \\
0 & \text{if $p\equiv3\pmod4$},
\end{cases}
\\
\sum_{k=0}^{p-1}\frac{(\frac12)_k^3}{k!^3} (-1)^k
\equiv a_2(p)\pmodModif{p^2},
\qquad
\sum_{k=0}^{p-1}\frac{(\frac12)_k^3}{k!^3} 4^k
\equiv a_3(p)\pmodModif{p^2},
\end{gather*}
where $a_1(n)$ denote the Fourier coefficients of the cusp form of weight~3 on $\Gamma_1(16)$,
\begin{gather*}
f_1(\tau)=\sum_{n=1}^\infty a_1(n)q^n=\eta(4\tau)^6=q\prod_{m=1}^\infty\big(1-q^{4m}\big)^6,
\end{gather*}
while
\begin{gather*}
f_2(\tau)=\sum_{n=1}^\infty a_2(n)q^n=\eta(\tau)^2\eta(2\tau)\eta(4\tau)\eta(8\tau)^2,\qquad f_3(\tau)=\sum_{n=1}^\infty a_3(n)q^n=\eta(2\tau)^3\eta(6\tau)^3
\end{gather*}
are the cusp forms on $\Gamma_1(8)$ and $\Gamma_1(12)$, respectively. In addition, on using some hypergeometric summations and \cite[Theorem~5]{RWZ15} we obtain
\begin{gather*}
{}_3F_2\biggl(\begin{matrix} \frac12, \, \frac12, \, \frac12 \\ 1, \, 1 \end{matrix}\biggm|1\biggr)
=\frac\pi{\Gamma(3/4)^4}=\frac{16L(f_1,2)}{\pi^2}=\frac{8L(f_1,1)}{\pi},\\
{}_3F_2\biggl(\begin{matrix} \frac12, \, \frac12, \, \frac12 \\ 1, \, 1 \end{matrix}\biggm|-1\biggr)
=\frac{\Gamma(1/8)^2\Gamma(3/8)^2}{2^{7/2}\pi^3}=\frac{12\sqrt2\,L(f_2,2)}{\pi^2}=\frac{12L(f_2,1)}{\pi},\\
\Re{}_3F_2\biggl(\begin{matrix} \frac12, \, \frac12, \, \frac12 \\ 1, \, 1 \end{matrix}\biggm|4\biggr)
=\frac{3\Gamma(1/3)^6}{2^{11/3}\pi^4}=\frac{12L(f_3,2)}{\pi^2}=\frac{4\sqrt3\,L(f_3,1)}{\pi}.
\end{gather*}
Also notice that algebraic transformations of underlying hypergeometric functions correspond to the `coincidences' of the type
\begin{gather*}
\sum_{k=0}^{p-1}\frac{(\frac12)_k(\frac13)_k(\frac23)_k}{k!^3} \left(\frac2{27}\right)^k
\equiv\sum_{k=0}^{p-1}\frac{(\frac12)_k^3}{k!^3} 4^k \equiv a_3(p)\pmodModif{p^2}
\end{gather*}
for $p>3$ and
\begin{gather*}
\sum_{k=0}^{p-1}\frac{(\frac12)_k(\frac14)_k(\frac34)_k}{k!^3}
\equiv\left(\frac{-4}p\right)\sum_{k=0}^{p-1}\frac{(\frac12)_k^3}{k!^3} (-1)^k \equiv\left(\frac{-4}p\right)a_2(p)\pmodModif{p^2}
\end{gather*}
for $p>2$. The last example is of importance in relation with the computation in~\cite{Sch16}.

\begin{remark}
\label{clausen}
Behind such examples in Observation~\ref{obs4}, there is Clausen's classical identity
\begin{gather}
{}_2F_1\biggl(\begin{matrix} r, \, 1-r \\ 1 \end{matrix}\biggm|z\biggr)^2
={}_3F_2\biggl(\begin{matrix} \frac12, \, r, \, 1-r \\ 1, \, 1 \end{matrix}\biggm|4z(1-z)\biggr)
\label{eqcl}
\end{gather}
valid in a neighbourhood of $z=0$. If we write the corresponding $\eps$-expansions \eqref{F01} and
\begin{align*}
\wt F(z;\eps) &=\frac1{\Gamma\big(\frac12\big) \Gamma(r) \Gamma(1-r)}
 \!\sum_{n=-\infty}^\infty \!\! \frac{\Gamma\big(\frac12+\eps+n\big) \Gamma(r+\eps+n) \Gamma(1-r+\eps+n)}{\Gamma(1+\eps+n)^3} (4z(1-z))^{n+\eps}\!\\
&=\wt F_0(z)+\wt F_1(z)\eps+\wt F_2(z)\eps^2+O\big(\eps^3\big) \qquad\text{as}\quad \eps\to0
\end{align*}
then $\wt F_0(z)=F_0(z)^2$ (as in \eqref{eqcl}) but also $\wt F_1(z)=F_0(z)F_1(z)$,
\begin{gather*}
\wt F_2(z)=\frac12\left(\frac\pi{\sin\pi r}\right)^2F_0(z)^2+\frac12F_1(z) =\frac12F_1(1-z)^2+\frac12F_1(z)^2.
\end{gather*}
The relations follow from the particular structure of the bilateral hypergeometric functions $F(z;\eps)$ and $\wt F(z;\eps)$, which we outlined in Section~\ref{bilateral}, and the following generalized Clausen identity:
\begin{gather}
2\wt F(z;\eps)\cos\pi\eps =F(z;\eps)^2e^{-\pi i\eps}\left(1-\frac{\sin^2\pi\eps}{\sin^2\pi r}\right) +F(z;0)^2e^{\pi i\eps}\label{eq:clausen}
\end{gather}
valid for all $\eps\in\mathbb R$. The identity \eqref{eq:clausen} follows from the fact that the hypergeometric differential equation for $\wt F(z;\eps)$ is the symmetric square of the differential equation for $F(z;\eps)$.

Finally, we would like to point out some heuristics about why modular instances of $K3$ surfaces with Picard rank~20 correspond to the CM cases of the underlying hypergeometric functions. Notice that the functional equation for $L(f,s)$ in the case of a modular form of weight~3 and level $\ell$ implies that, for the critical values, $L(f,2)/L(f,1)=\pm2\pi/\sqrt\ell$. If we expect that a hypergeometric~$_3F_2$ function is linked to
a~modular~$K3$ surface (with Picard rank~20), then we must have $\wt F_2(z)/\big(\pi\wt F_1(z)\big)$ to be of the form $\sqrt l \mathbb Q$ for some positive integer $\ell$. With the help of the generalized Clausen identity we then conclude that the quantity
\begin{gather*}
\tau=\tau(z)=-\frac{iF_1(z)}{2\pi F_0(z)}=\frac{iF_0(1-z)}{2\sin\pi r\,F_0(z)}
\end{gather*}
must be an imaginary quadratic irrationality, hence its functional inversion~-- the modular function $z=z(\tau)$ admits a singular modulus value at this point. The fact that $z(\tau)$ is a modular parametrization of the corresponding hypergeometric function
\begin{gather*}
F_0(z)={}_2F_1\biggl(\begin{matrix} r, \, 1-r \\ 1 \end{matrix}\biggm|z\biggr)
\end{gather*}
for each $r\in\big\{\frac12,\frac13,\frac14,\frac16\big\}$ is classical~-- see, for example, \cite[p.~91]{Be97}; one also has
\begin{gather*}
\frac1{2\pi i}\,\frac{\d z}{\d\tau}=z(1-z)F_0(z),
\end{gather*}
the result already known to Ramanujan \cite[Chapter~33]{Be97}, \cite{Co09}.

A different way to explain the modularity of $K3$ surfaces with Picard number~20 is kindly communicated to us by N.~Yui:
Such $K3$ surfaces are all motivically modular in the sense that the lattice of transcendental cycles is of rank 2 and corresponds
to a modular form of weight~3 with character for some congruence subgroup of ${\rm PSL}_2(\mathbb Z)$.
They are all of CM type as the endomorphism algebra of the transcendental lattice is an imaginary quadratic field over $\mathbb Q$.
In particular, this means that the underlying hypergeometric functions are also of CM type.
\end{remark}

Another interesting instance corresponds to choosing $z=1$ in the hypergeometric series
\begin{gather*}
F_0(z)={}_6F_5\biggl(\begin{matrix} \frac12, \, \frac12, \, \frac12, \, \frac12, \, \frac12, \, \frac12 \\ 1, \, 1, \, 1, \, 1, \, 1 \end{matrix}\biggm|z\biggr)
=\sum_{k=0}^\infty\frac{(\frac12)_k^6}{k!^6} z^k
\end{gather*}
related to a Calabi--Yau fivefold~-- a complete intersection of six degree~2 surfaces in $\mathbb P^{12}$;
the associated Hodge structure for each fiber $z$ of the family can be conjecturally computed with a~help of the hypergeometric motives~\cite{RV17}.
Consider the newform
\begin{align*}
g(\tau) &=\sum_{n=1}^\infty b(n)q^n =q+20q^3-74q^5-24q^7+157q^9+124q^{11}+\dotsb \\
&=\eta(2\tau)^{12}+32\eta(2\tau)^4\eta(8\tau)^8
\end{align*}
of weight 6 on $\Gamma_0(8)$. Its coefficients satisfy Weil's bound $|b(p)|\le2p^{5/2}$ and numerics suggest that
\begin{gather}
\sum_{k=0}^{p-1}\frac{(\frac12)_k^6}{k!^6}\equiv b(p)\pmodModif{p^5} \label{e-mort}
\end{gather}
is true for all primes $p>2$. The explicit expression for $g(\tau)$ was kindly informed to us by J.~Wan who also noticed its historical cast in~\cite{Gl07} (see the last column of the table on p.~56 there). As we learned later, the conjecture \eqref{e-mort} was reported in \cite{FOP04} and attributed to E.~Mortenson; it is now shown to be true modulo $p^3$ in the joint work \cite{OSZ18} with R.~Osburn and A.~Straub. Numerically, the Taylor $\eps$-expansion
\begin{align*}
\frac1{\Gamma(\frac12)^6}\sum_{n=-\infty}^\infty\frac{\Gamma(\frac12+\eps+n)^6}{\Gamma(1+\eps+n)^6} z^{n+\eps}=\sum_{k=0}^5F_k(z)\eps^k+O\big(\eps^6\big) \qquad\text{as}\quad \eps\to0
\end{align*}
can be related, at $z=1$, to the critical $L$-values as follows:
\begin{gather*}
\frac{L(g,1)}{F_1(1)}=-\frac18, \qquad \frac{L(g,2)}{F_2(1)}=\frac1{32}, \qquad \frac{L(g,3)}{F_3(1)}=-\frac3{448}, \\
\frac{L(g,4)}{F_4(1)}=\frac1{640} \qquad\text{and}\qquad \frac{L(g,5)}{F_5(1)}=-\frac5{12032}.
\end{gather*}
As pointed out to us by F.~Rodriguez Villegas and D.~Roberts the related hypergeometric motive is also linked to the modular form $f(\tau)$ from the introduction, defined in \eqref{e01a}. Armed by this hint, we have found the related instances
\begin{gather*}
\sum_{k=0}^{p-1}(4k+1)\frac{(\frac12)_k^6}{k!^6}\equiv pa(p)\pmodModif{p^4} \quad\text{for}\ p>2
\end{gather*}
proved in \cite[Theorem 1.2]{Lo11} and
\begin{gather*}
\sum_{k=0}^\infty(4k+1)\frac{(\frac12)_k^6}{k!^6}=\frac{32}{\pi^2} L(f,1)
\end{gather*}
established in \cite[equation~(33)]{RWZ15}.

Our final~-- and personal favourite~-- family of examples is about known Rama\-nujan(-type) formulas \cite{Zu08} for $1/\pi$, $1/\pi^2$ and their generalizations. Those fit a general picture highlighted in the observations above, except that the modular form $f(\tau)$ is replaced by a quadratic character so that a critical $L$-value $L(f,m)$ is replaced by the critical value of the corresponding Dirichlet $L$-series. This is transparent from supercongruence observations in~\cite{Zu09} and, in addition, from a~noncongruence (bilateral) counterpart experimentally discovered by J.~Guillera in~\cite{Gu16} (see also the related prequel~\cite{GR14}).

\subsection*{Acknowledgements}

Feedback of many colleagues has been extremely helpful in preparation of this manuscript. I~would like to thank Frits Beukers, Henri Cohen, Jes\'us Guillera, G\"unter Harder, Ling Long, Anton Mellit, Alexei Panchishkin, David Roberts, Emanuel Scheidegger, Duco van Straten, Alexander Varchenko, Fernando Rodriguez Villegas, John Voight, James Wan, Noriko Yui and Don Zagier for their comments and responses to my questions.
Special thanks are expressed to Vasily Golyshev for his clarification to me the link between the critical $L$-values and the corresponding hypergeometrics, which underlies so-called gamma structures~\cite{GM14}, and to Michael Somos for his powerful help in making some entries in Table~\ref{tab1} explicit. Finally, I am indebted to the anonymous referees for several helpful comments and corrections.

This note grew up from the author's talk at the BIRS Workshop ``Modular Forms in String Theory'' held in September 2016, and related discussions there. Later parts of this work were performed during the author's visits in research institutions whose hospitality and scientific atmosphere were crucial to success of the project. I thank the staff of the following institutes for providing such excellent conditions for research: BIRS (Banff, Canada, September 2016); MATRIX (Creswick, Australia, January 2017); ESI (Vienna, Austria, March 2017); MPIM (Bonn, Germany, December 2016 and July--August 2017); HIM (Bonn, Germany, March--April 2018).

The author is partially supported by Laboratory of Mirror Symmetry NRU HSE, RF go\-vernment grant, ag.\ no.\ 14.641.31.0001.

\pdfbookmark[1]{References}{ref}
\LastPageEnding


\begin{thebibliography}{99}
\footnotesize\itemsep=0pt

\bibitem{AO00}
Ahlgren S., Ono K., A {G}aussian hypergeometric series evaluation and {A}p\'ery
 number congruences, \href{https://doi.org/10.1515/crll.2000.004}{\textit{J.~Reine Angew. Math.}} \textbf{518} (2000),
 187--212.

\bibitem{Be97}
Berndt B.C., Ramanujan's notebooks, {P}art~{V}, \href{https://doi.org/10.1007/978-1-4612-1624-7}{Springer-Verlag}, New York,
 1998.

\bibitem{Be87}
Beukers F., Another congruence for the {A}p\'ery numbers, \href{https://doi.org/10.1016/0022-314X(87)90025-4}{\textit{J.~Number
 Theory}} \textbf{25} (1987), 201--210.

\bibitem{BCM15}
Beukers F., Cohen H., Mellit A., Finite hypergeometric functions, \href{https://doi.org/10.4310/PAMQ.2015.v11.n4.a2}{\textit{Pure
 Appl. Math.~Q.}} \textbf{11} (2015), 559--589, \href{https://arxiv.org/abs/1505.02900}{arXiv:1505.02900}.

\bibitem{Co09}
Cooper S., Inversion formulas for elliptic functions, \href{https://doi.org/10.1112/plms/pdp007}{\textit{Proc. Lond. Math.
 Soc.}} \textbf{99} (2009), 461--483.

\bibitem{Da70}
Damerell R.M., {$L$}-functions of elliptic curves with complex
 multiplication.~{I}, \href{https://doi.org/10.4064/aa-17-3-287-301}{\textit{Acta Arith.}} \textbf{17} (1970), 287--301.

\bibitem{Ev00}
Evans R., Review of \cite{AO00}, \textit{MathSciNet}, MR1739404 (2001c:11057), Amer.
 Math. Soc., Providence, RI, available at
 \url{http://www.ams.org/mathscinet-getitem?mr=1739404}.

\bibitem{FOP04}
Frechette S., Ono K., Papanikolas M., Gaussian hypergeometric functions and
 traces of {H}ecke operators, \href{https://doi.org/10.1155/S1073792804132522}{\textit{Int. Math. Res. Not.}} \textbf{2004}
 (2004), 3233--3262.

\bibitem{Gl07}
Glaisher J.W.L., On the representations of a number as the sum of two, four,
 six, eight, ten, and twelve squares, \textit{Quart.~J. Pure Appl. Math.} \textbf{38} (1906),
 1--62.

\bibitem{GM14}
Golyshev V., Mellit A., Gamma structures and {G}auss's contiguity,
 \href{https://doi.org/10.1016/j.geomphys.2013.12.007}{\textit{J.~Geom. Phys.}} \textbf{78} (2014), 12--18, \href{https://arxiv.org/abs/0902.2003}{arXiv:0902.2003}.

\bibitem{Gr87}
Greene J., Hypergeometric functions over finite fields, \href{https://doi.org/10.2307/2000329}{\textit{Trans. Amer.
 Math. Soc.}} \textbf{301} (1987), 77--101.

\bibitem{Gu16}
Guillera J., Bilateral sums related to {R}amanujan-like series,
 \href{https://arxiv.org/abs/1610.04839}{arXiv:1610.04839}.

\bibitem{GR14}
Guillera J., Rogers M., Ramanujan series upside-down, \href{https://doi.org/10.1017/S1446788714000147}{\textit{J.~Aust. Math.
 Soc.}} \textbf{97} (2014), 78--106, \href{https://arxiv.org/abs/1206.3981}{arXiv:1206.3981}.

\bibitem{Ki06}
Kilbourn T., An extension of the {A}p\'ery number supercongruence, \href{https://doi.org/10.4064/aa123-4-3}{\textit{Acta
 Arith.}} \textbf{123} (2006), 335--348.


\bibitem{LMFDB}
{The LMFDB Collaboration}, The $L$-functions and modular forms database,
 available at \url{http://www.lmfdb.org}.

\bibitem{Lo11}
Long L., Hypergeometric evaluation identities and supercongruences,
 \href{https://doi.org/10.2140/pjm.2011.249.405}{\textit{Pacific~J. Math.}} \textbf{249} (2011), 405--418, \href{https://arxiv.org/abs/0912.0197}{arXiv:0912.0197}.

\bibitem{LTYZ17}
Long L., Tu F.-T., Yui N., Zudilin W., Supercongruences for rigid hypergeometric
 {C}alabi--{Y}au threefolds, \href{https://arxiv.org/abs/1705.01663}{arXiv:1705.01663}.

\bibitem{Mc12}
McCarthy D., Extending {G}aussian hypergeometric series to the {$p$}-adic
 setting, \href{https://doi.org/10.1142/S1793042112500844}{\textit{Int.~J. Number Theory}} \textbf{8} (2012), 1581--1612,
 \href{https://arxiv.org/abs/1204.1574}{arXiv:1204.1574}.

\bibitem{OSZ18}
Osburn R., Straub A., Zudilin W., A modular supercongruence for {$_6F_5$}: an
 {A}p\'ery-like story, \textit{Ann. Inst. Fourier (Grenoble)}, {t}o appear,
 \href{https://arxiv.org/abs/1701.04098}{arXiv:1701.04098}.

\bibitem{RRV17}
Roberts D., Rodriguez~Villegas F., Hypergeometric supercongruences,
 \href{https://arxiv.org/abs/1803.10834}{arXiv:1803.10834}.

\bibitem{RV03}
Rodriguez~Villegas F., Hypergeometric families of {C}alabi--{Y}au manifolds, in
 Calabi--{Y}au Varieties and Mirror Symmetry ({T}oronto, {ON}, 2001),
 \textit{Fields Inst. Commun.}, Vol.~38, Amer. Math. Soc., Providence, RI,
 2003, 223--231.

\bibitem{RV17}
Rodriguez~Villegas F., Hypergeometric motives, {L}ecture notes, 2017.

\bibitem{RWZ15}
Rogers M., Wan J.G., Zucker I.J., Moments of elliptic integrals and critical
 {$L$}-values, \href{https://doi.org/10.1007/s11139-014-9584-5}{\textit{Ramanujan~J.}} \textbf{37} (2015), 113--130,
 \href{https://arxiv.org/abs/1303.2259}{arXiv:1303.2259}.

\bibitem{Sch16}
Scheidegger E., Analytic continuation of hypergeometric functions in the
 resonant case, \href{https://arxiv.org/abs/1602.01384}{arXiv:1602.01384}.

\bibitem{Sh76}
Shimura G., The special values of the zeta functions associated with cusp
 forms, \href{https://doi.org/10.1002/cpa.3160290618}{\textit{Comm. Pure Appl. Math.}} \textbf{29} (1976), 783--804.

\bibitem{Sh77}
Shimura G., On the periods of modular forms, \href{https://doi.org/10.1007/BF01391466}{\textit{Math. Ann.}} \textbf{229}
 (1977), 211--221.

\bibitem{Si86}
Silverman J.H., The arithmetic of elliptic curves, \href{https://doi.org/10.1007/978-0-387-09494-6}{\textit{Graduate Texts in
 Mathematics}}, Vol.~106, 2nd~ed., Springer, Dordrecht, 2009.

\bibitem{Sl66}
Slater L.J., Generalized hypergeometric functions, Cambridge University Press,
 Cambridge, 1966.

\bibitem{St06}
Stienstra J., Mahler measure variations, {E}isenstein series and instanton
 expansions, in Mirror Symmetry.~{V}, \textit{AMS/IP Stud. Adv. Math.},
 Vol.~38, Amer. Math. Soc., Providence, RI, 2006, 139--150,
 \href{https://arxiv.org/abs/math.NT/0502193}{math.NT/0502193}.

\bibitem{SB85}
Stienstra J., Beukers F., On the {P}icard--{F}uchs equation and the formal
 {B}rauer group of certain elliptic {$K3$}-surfaces, \href{https://doi.org/10.1007/BF01455990}{\textit{Math. Ann.}}
 \textbf{271} (1985), 269--304.

\bibitem{Sw15}
Swisher H., On the supercongruence conjectures of van {H}amme, \href{https://doi.org/10.1186/s40687-015-0037-6}{\textit{Res.
 Math. Sci.}} \textbf{2} (2015), Art.~18, 21~pages, \href{https://arxiv.org/abs/1504.01028}{arXiv:1504.01028}.

\bibitem{vH97}
van Hamme L., Some conjectures concerning partial sums of generalized
 hypergeometric series, in {$p$}-Adic Functional Analysis ({N}ijmegen, 1996),
 \textit{Lecture Notes in Pure and Appl. Math.}, Vol.~192, Editors W.H.~Schikhof, C.~Perez-Garcia, J.~Kakol, Dekker, New York, 1997, 223--236.

\bibitem{Ve10}
Verrill H.A., Congruences related to modular forms, \href{https://doi.org/10.1142/S1793042110003587}{\textit{Int.~J. Number
 Theory}} \textbf{6} (2010), 1367--1390.

\bibitem{Za16}
Zagier D.B., Arithmetic and topology of differential equations, in Proceedings of the Seventh European Congress of Mathematics (Berlin, July 18--22, 2016), Editors V.~Mehrmann, M.~Skutella, \href{https://doi.org/10.4171/176-1/33}{European Mathematical Society}, Berlin, 2018, 717--776.

\bibitem{Ze92}
Zeilberger D., Gauss's {${}_2F_1(1)$} cannot be generalized to {${}_2F_1(x)$},
 \href{https://doi.org/10.1016/0377-0427(92)90211-F}{\textit{J.~Comput. Appl. Math.}} \textbf{39} (1992), 379--382.

\bibitem{Zu08}
Zudilin W., Ramanujan-type formulae for {$1/\pi$}: a second wind?, in Modular
 Forms and String Duality, \textit{Fields Inst. Commun.}, Vol.~54, Amer. Math.
 Soc., Providence, RI, 2008, 179--188, \href{https://arxiv.org/abs/0712.1332}{arXiv:0712.1332}.

\bibitem{Zu09}
Zudilin W., Ramanujan-type supercongruences, \href{https://doi.org/10.1016/j.jnt.2009.01.013}{\textit{J.~Number Theory}}
 \textbf{129} (2009), 1848--1857, \href{https://arxiv.org/abs/0805.2788}{arXiv:0805.2788}.

\end{thebibliography}
\end{document}